\documentclass[11pt]{amsart}

\usepackage{graphicx}
\usepackage{subfig}
\usepackage{amsmath,amsthm,amssymb}
\usepackage{esint}
\usepackage{mathrsfs}
\usepackage{appendix}

\usepackage{hyperref}

\usepackage{color}
\usepackage{array}
\usepackage{url}

\usepackage{float}

\newtheorem{theorem}{Theorem}
\newtheorem*{theorem*}{Theorem}

\newtheorem{proposition}[theorem]{Proposition}

\theoremstyle{definition}
\newtheorem{definition}[theorem]{Definition}

\theoremstyle{remark}
\newtheorem{remark}[theorem]{Remark}

\usepackage{mathtools}
\usepackage{appendix}
\usepackage{enumerate}

\usepackage[numbers,sort&compress]{natbib}

\DeclareMathOperator{\Span}{span}

\allowdisplaybreaks

\title[A condition for the absence of anomalous dissipation]{On the absence of anomalous dissipation for the Navier-Stokes equations with Navier boundary conditions: a sufficient condition}

\author[C. Bardos]{Claude Bardos}
\address[C. Bardos]{Laboratoire J.-L. Lions, BP187, 75252 Paris Cedex 05, France; Wolfgang Pauli Institute - Inst. CNRS Pauli IRL 2842 c/o Faculty of Mathematics, University of Vienna, Oskar-Morgenstern-Platz 1, A-1090 Wien, Austria.}
\email{\tt claude.bardos@gmail.com \pagebreak[2] }

\author[D. W. Boutros]{Daniel W. Boutros}
\address[D. W. Boutros]{Department of Applied Mathematics and Theoretical Physics, University of Cambridge, Cambridge CB3 0WA UK.}
\email{\tt dwb42@cam.ac.uk}

\author[E. S. Titi]{Edriss S. Titi}
\address[E. S. Titi]{Department of Applied Mathematics and Theoretical Physics, University of Cambridge, Cambridge CB3 0WA UK; Department of Mathematics, Texas A\&M University, College Station, TX 77843-3368, USA; also Department of Computer Science and Applied Mathematics, Weizmann Institute of Science, Rehovot 76100, Israel. \pagebreak[2] }
\email{\tt Edriss.Titi@maths.cam.ac.uk}
\email{\tt titi@math.tamu.edu}



\date{March 18, 2026}

\begin{document}

\maketitle

\begin{abstract}
We consider the three-dimensional incompressible Navier-Stokes equations in a bounded domain with Navier boundary conditions. We provide a sufficient condition for the absence of anomalous energy dissipation without making assumptions on the behaviour of the corresponding pressure near the boundary or the existence of a strong solution to the incompressible Euler equations with the same initial data. We establish our result by using our recent regularity results for the pressure corresponding to weak solutions of the incompressible Euler equations [Arch. Ration. Mech. Anal., 249 (2025), 28]. 
\end{abstract}

\vspace{8mm}

{\centering \textit{This paper is dedicated to Professor Peter Constantin, on the occasion of his 75th birthday, as a token of friendship and admiration for his contributions to research in partial differential equations and fluid mechanics.} \par}

\bigskip

\section{Introduction}
We consider the incompressible three-dimensional Navier-Stokes equations in a bounded domain $\Omega$ with $\partial \Omega \in C^4$, which, for a given viscosity $\nu > 0$, are described by the system
\begin{equation} \label{navierstokes}
 \partial_t u_\nu + \nabla \cdot(u_\nu\otimes u_\nu)-\nu \Delta u_\nu+ \nabla p_\nu =0, \quad \nabla\cdot u_\nu=0,
\end{equation}
where $u_\nu : \Omega \times [0,T] \rightarrow \mathbb{R}^3$ is the velocity field and $p_\nu : \Omega \times [0,T] \rightarrow \mathbb{R}$ is the pressure, which are the unknowns. Moreover, system \eqref{navierstokes} is supplemented by the following initial data
\begin{equation}
u_\nu \lvert_{t = 0} = u_{\nu,0},
\end{equation}
where $u_{\nu,0} \in L^2 (\Omega)$, with $\nabla \cdot u_{\nu,0}=0$ (in a weak sense). We define $n$ to be the outward normal vector to the boundary $\partial \Omega$. 

On the boundary and for a given well-defined vector field $v \in C^{0,\alpha} (\Omega)$ (for $\alpha \in (0,1)$), we will denote its projection onto the tangential components as follows
\begin{equation}
(v )_\tau \coloneqq \big[ v - (v \cdot n) n \big]_{\partial \Omega}.
\end{equation}
In addition, we will supplement system \eqref{navierstokes} with Navier boundary conditions, which consist of a no-flux (impermeability) boundary condition and a complementary boundary condition $C$ (defined in equation \eqref{navierboundarycondition}, below)
\begin{equation} \label{boundaryconditions}
u_\nu\cdot n =0\quad  \text{ and } \quad  C(u_\nu)=0 \quad \hbox{ on } \partial \Omega.
\end{equation}
Similarly to \cite{bardos2013,gie}, we will consider a generalised (tangential) Navier boundary condition $C$ of the type
\begin{equation} \label{navierboundarycondition}
C(u_\nu) \coloneqq  \left( D(u_\nu) \cdot n \right)_\tau + \lambda (\nu, x) \cdot u_\nu  = 0 \quad \hbox{ on } \partial \Omega,
\end{equation}
where $D(v) \coloneqq \frac{1}{2} [\nabla v + (\nabla v)^T]$ denotes the symmetric part of the gradient of the vector field $v$, (for the sake of clarity) we observe that $D(u_\nu) \cdot n$ denotes the (usual) product between a matrix and a vector field, the tensor-valued function $\lambda : (0,\infty) \times \Omega \rightarrow \mathbb{R}^{3 \times 3}$ is assumed to be a given continuous and nonnegative semidefinite function. This class of boundary conditions includes the Robin boundary conditions or boundary conditions involving the vorticity or the stress tensor (in the latter case see, e.g., \cite{beiraodaveiga2010}). 

Now we recall the identity (for $v \in H^2 (\Omega)$ with $(v \cdot n) \lvert_{\partial \Omega} = 0$)
\begin{equation} \label{vorticityidentity}
(2 D (v) \cdot n - (\nabla \times v) \times n)_\tau = - 2 (D(n) \cdot v)_\tau \quad \hbox{ on } \partial \Omega,
\end{equation}
which holds in $H^{1/2} (\partial \Omega)$ in the sense of trace on the boundary. The proof of identity \eqref{vorticityidentity} can for example be found in \cite[Lemma 3.10]{xiao2013}, see also \cite[Lemma B.1]{gie}. Therefore, thanks to \eqref{vorticityidentity} we can rewrite the Navier boundary condition \eqref{navierboundarycondition} as follows
\begin{equation}
(\omega_\nu \times n)_\tau + 2 \lambda (\nu, x) \cdot u_\nu - 2 \big( D(n) \cdot u_\nu \big)_\tau = 0 \quad \hbox{ on } \partial \Omega,
\end{equation}
where $\omega_\nu = \nabla \times u_\nu$ is the vorticity. If one chooses $\lambda = D(n)$ (which is not necessarily nonnegative semidefinite), the Navier boundary condition \eqref{navierboundarycondition} is equivalent to the following vorticity boundary condition (in three dimensions)
\begin{equation}
\omega_\nu \times n \lvert_{\partial \Omega} = 0.
\end{equation}
In the case of two dimensions one has instead of \eqref{vorticityidentity} the following identity
\begin{equation} \label{vorticityidentity2D}
(D(u_\nu) \cdot n) \cdot \tau = \frac{1}{2} \omega_\nu - \kappa (u_\nu \cdot \tau),
\end{equation}
where $\kappa$ is the curvature of the boundary $\partial \Omega$. The proof of identity \eqref{vorticityidentity2D} can be found in \cite[Lemma 2.1]{clopeau}, as well as \cite[~Collorary 4.2]{kelliher}. Observe that in the two-dimensional case the function $\lambda$ is scalar-valued, therefore if we choose $\lambda = \kappa$ in \eqref{navierboundarycondition} together with \eqref{vorticityidentity2D} we end up with the following boundary condition (in two dimensions)
\begin{equation}
\omega_\nu \lvert_{\partial \Omega} = 0.
\end{equation}

Formally, if we take the vanishing viscosity limit $\nu \rightarrow 0$ of the Navier-Stokes equations \eqref{navierstokes} we obtain the Euler equations
\begin{equation} \label{eulerequations1}
 \partial_t u + \nabla \cdot(u \otimes u ) + \nabla p =0, \quad \nabla\cdot u = 0,
\end{equation}
which is supplemented by the following impermeability boundary condition
\begin{equation} \label{eulerequations2}
(u \cdot n) \lvert_{\partial \Omega} = 0.
\end{equation}
The dissipation anomaly (also known as the Kolmogorov zeroth law of turbulence) is one of the most fundamental properties of turbulent flows and is a crucial feature of the theory of turbulence. In this context, we assert that the Navier-Stokes equations exhibit anomalous dissipation if there is a sequence of Leray-Hopf weak solutions $\{ u_{\nu_k} \}$ of the Navier-Stokes equations (for a sequence of viscosities $\nu_k \rightarrow 0$) such that (cf. Definition \ref{lerayhopfdefinition})
\begin{align*}
\liminf_{\nu_k \rightarrow 0} \bigg[&2 \nu_k \int_0^T \lVert D (u_{\nu_k}) (\cdot, t) \rVert_{L^2 (\Omega)}^2 dt \\
+ &2 \nu_k \sum_{i,j=1}^3 \int_0^T \int_{\partial \Omega} \lambda_{ij} (\nu_k,x) (u_{\nu_k})_i (u_{\nu_k})_j dx dt\bigg] > 0.
\end{align*}

The purpose of this contribution is to establish a sufficient condition for the absence of anomalous dissipation in the case of Navier boundary conditions, with no regularity assumptions on the corresponding pressure (near the boundary). In particular, in order to prove this result we employ recently established regularity results for the pressure for Hölder continuous weak solutions of the Euler equations (cf. \cite{bardos2023}). 

To be more precise, in \cite{bardos2023} it was shown that Hölder regularity propagates from the velocity field to the pressure for weak solutions of the three-dimensional incompressible Euler equations, specifically if the velocity field of a weak solution $u \in L^\infty ((0,T); C^{0,\alpha} (\Omega))$, for $\alpha \in (0,1)$, then the corresponding pressure $p \in L^\infty ((0,T); C^{0,\alpha} (\Omega))$ (after the same result was established in 2D in \cite{bardos2021}, see also \cite{derosa2023,derosa2024,li2024}). 

We will utilise the results from \cite{bardos2023} in this paper to show a sufficient condition for the convergence of Leray-Hopf weak solutions of the Navier-Stokes equations to a weak solution of the incompressible Euler equations and the absence of anomalous energy dissipation. Most importantly, we observe that in \cite{bardos2023} it was established that the boundary condition for the pressure for weak solutions of the Euler equations, which is rigorously derived in \cite{bardos2023} from the weak formulation, is different from the standard classical one for strong solutions. The nonequivalence is further demonstrated by means of an explicitly constructed example in \cite{bardos2023}.

The literature regarding the dissipation anomaly (and the associated Onsager conjecture) as well as the vanishing viscosity limit is very substantial. Consequently, the following overview will therefore be necessarily incomplete in reviewing all the relevant work. In the case of no-slip Dirichlet boundary conditions for the Navier-Stokes equations, the vanishing of the total (viscous) energy dissipation (as the viscosity goes to zero) in a thin boundary layer (the so-called Prandtl-Von Karman layer) is equivalent to the validity of the vanishing viscosity limit. That is, the solutions of the Navier-Stokes equations converge, as the viscosity tends to zero, to the corresponding solution of the Euler equations, under the assumption of the existence of a strong solution of the Euler equations. This result is known as the Kato criterion and was established in \cite{kato} (see also the survey \cite{bardos2013}). 

A related criterion was then proved in \cite{wang}, which only required the vanishing of the part of the energy dissipation due to tangential derivatives of the velocity field (in a slightly thicker boundary layer than the one in \cite{kato}). In \cite{drivas2018,chenkato} the regularity requirements on the limiting background Euler solution were lowered from $C^1(\Omega)$ to $C^{1/3+} (\Omega)$. In addition, under second-order structure function scaling assumptions it was shown in \cite{constantinremarks,drivas2019,bardos2013} that Leray-Hopf weak solutions of the Navier-Stokes equations converge, as the viscosity tends to zero, to (potentially dissipative and weak) solutions of the Euler equations, both for no-slip and Navier boundary conditions.

In \cite{titi2019} (see also \cite{drivas2018}) a sufficient condition for the ruling out of anomalous energy dissipation was established (again in the case of no-slip boundary conditions), which involved assumptions of interior Hölder regularity and continuity of the normal component of the energy flux. Note that these assumptions are consistent with the formation of a boundary layer in the setting of no-slip boundary conditions. We also mention that in \cite{yang} uniform bounds (in viscosity) on the energy dissipation were obtained in the case of nonhomogeneous boundary conditions. 

It is worth stressing that many results have been obtained regarding the vanishing viscosity limit of the Navier-Stokes equations with Navier boundary conditions, see, for example, \cite{aydin,clopeau,beiraodaveiga2010,gie,kelliher,bardos1972,lopes2005,masmoudi2012,maekawa}, and references therein. Furthermore, results on the boundary layer expansion for the case of Navier boundary conditions can be found in \cite{wang2010,iftimie2011}. Finally, we remark again that the study of the vanishing viscosity limit of the Navier-Stokes equations in the case of no-slip boundary conditions has received significant attention in the literature, results in this direction can be found in \cite{sammartino,kukavica2020,constantin2015,maekawa2014,bardos2022,lopes2008,constantin2017,vasseur2024}, and references therein. An overview of results for the no-slip case can be found in \cite{maekawa}.

\section{Preliminary results and weak formulation of the Euler and Navier-Stokes equations}
In this section we will provide precise definitions of the notions of weak solution for the Euler and Navier-Stokes equations, respectively, that we will use in this paper. In particular, we will show the equivalence of two different versions of the weak formulation for the Euler equations. This result will be used later in the proof of a sufficient condition for the convergence of Leray-Hopf weak solutions of the Navier-Stokes equations to a weak solution of the Euler equations.

First we will give a definition of a Leray-Hopf weak solution for the Navier-Stokes equations \eqref{navierstokes} subject to Navier boundary conditions \eqref{boundaryconditions} and \eqref{navierboundarycondition}, which will be based on \cite{kelliher2025} (the definition for the case of Dirichlet boundary conditions can for example be found in \cite{constantinbook}). In the following definition we take into account the fact that for any divergence-free distributional vector field $v$ one has the following equality in the sense of distributions: $\Delta v = 2 \nabla \cdot D(v)$.
\begin{definition} \label{lerayhopfdefinition}
Let $u_\nu \in C_w ([0,T]; L^2 (\Omega)) \cap L^2 ((0,T); H^1 (\Omega))$ be weakly divergence-free such that $(u_\nu \cdot n) \lvert_{\partial \Omega} = 0$ in $L^\infty ((0,T); H^{-1/2} (\partial \Omega))$. Moreover, let $u_{\nu,0} \in L^2 (\Omega)$ also be a weakly divergence-free vector field such that $(u_{\nu,0} \cdot n) \lvert_{\partial \Omega} = 0$. We call $u_\nu$ a Leray-Hopf weak solution of the Navier-Stokes equations \eqref{navierstokes} with Navier boundary conditions \eqref{boundaryconditions} and \eqref{navierboundarycondition} for a given initial datum $u_{\nu,0}$ if it satisfies the following:
\begin{itemize}
    \item For all divergence-free $\Psi \in \mathcal{D} ([0,T) \times \overline{\Omega} ; \mathbb{R}^3)$ with $(\Psi \cdot n) \lvert_{\partial \Omega} = 0$ it holds that (cf. equation (4.6) in \cite{kelliher2025})
    \begin{align}
    &\int_0^T \int_{\Omega} \bigg[u_\nu \cdot \partial_t \Psi + (u_\nu \otimes u_\nu) : \nabla \Psi - 2 \nu D (u_\nu) : \nabla \Psi  \bigg] dx dt \label{weakformulationnavierstokes} \\
    &+ \int_{\Omega} u_{\nu,0} \cdot \Psi (x,0) dx - 2 \nu \int_0^T \int_{\partial \Omega} (\lambda (\nu,x) \cdot u_\nu) \cdot \Psi dx dt = 0. \nonumber
    \end{align}
    \item The velocity vector field $u_\nu$ obeys the following energy inequality (see \cite{bardos2013} and also \cite[equation (4.8)]{kelliher2025})
    \begin{align} 
    &\frac{1}{2} \lVert u_\nu (\cdot, t) \rVert_{L^2}^2 + 2 \nu \int_0^t \lVert D (u_\nu) (\cdot, t') \rVert_{L^2}^2 dt' \nonumber \\
    &+ 2 \nu \sum_{i,j=1}^3 \int_0^t \int_{\partial \Omega} \lambda_{ij} (\nu,x) (u_{\nu})_i (u_{\nu})_j dx dt' \leq \frac{1}{2} \lVert u_{\nu,0} \rVert_{L^2}^2. \label{lerayhopf2}
    \end{align}
    for all $t \in [0,T]$. 
\end{itemize}
\end{definition}
\begin{remark}
We observe that in Definition \ref{lerayhopfdefinition}, the Navier boundary condition \eqref{navierboundarycondition} has been included as part of the weak formulation \eqref{weakformulationnavierstokes} (differently from the case of Dirichlet boundary conditions) rather than being a separate requirement in the definition. By using the $L^2 ((0,T); H^1 (\Omega))$ regularity assumption on the velocity vector field for Leray-Hopf weak solutions (in the sense of Definition \ref{lerayhopfdefinition}), one is still able to interpret equation \eqref{navierboundarycondition} in the sense of trace in $L^2 ((0,T); H^{-1/2} (\partial \Omega))$, see for example \cite[Lemma 2.3 and 2.4]{amrouche}. One can therefore deduce from the weak formulation in equation \eqref{weakformulationnavierstokes} that boundary condition \eqref{navierboundarycondition} holds in the sense of trace. For the purposes of this paper, we will require that the Navier boundary condition \eqref{navierboundarycondition} holds as part of the weak formulation \eqref{weakformulationnavierstokes}, rather than in the sense of trace.
\end{remark}
\begin{remark}
To the best knowledge of the authors, the global existence of Leray-Hopf weak solutions for the incompressible Navier-Stokes equations with Navier boundary conditions of the type \eqref{boundaryconditions}-\eqref{navierboundarycondition} has not been established in the literature to this degree of generality. The case $\lambda \equiv 0$ was treated in \cite{falocchi,watanabe}. In the case $\lambda = \tilde{\lambda} I$ (for some scalar function $\tilde{\lambda} (\nu, x)$), a sketch of the proof of the global existence of Leray-Hopf weak solutions was given in \cite{iftimie2011}. A full proof of this case was given in \cite{kelliher2025}. Results in the 2D setting can be found in \cite{clopeau}. For the sake of completeness, in Appendix \ref{lerayhopfappendix} we will provide an outline of the proof of the global existence of Leray-Hopf weak solutions for the boundary conditions \eqref{boundaryconditions}-\eqref{navierboundarycondition}, by following the approach in \cite{kelliher2025} (see also \cite{gie}).
\end{remark}
We will also provide a definition of a weak solution to the incompressible Euler equations.
\begin{definition} \label{eulerweaksolutiondefinition}
Let $\alpha \in (0,1)$. We call a pair $(u,p) \in C([0,T]; C^{0,\alpha} (\Omega)) \times L^\infty ((0,T); C^{0,\alpha} (\Omega))$ a weak solution of the incompressible Euler equations \eqref{eulerequations1}-\eqref{eulerequations2} with initial data $u_0 \in C^{0,\alpha} (\Omega)$ (which we assume to be weakly divergence-free), if for all $\phi \in \mathcal{D} (\Omega \times [0,T); \mathbb{R}^3)$ and $\varphi \in \mathcal{D} (\Omega \times (0,T); \mathbb{R})$ we have
\begin{align}
&\int_0^T \int_\Omega \bigg[u \cdot \partial_t \phi + (u \otimes u) : \nabla \phi + p \nabla \cdot \phi \bigg] dx dt + \int_{\Omega} u_0 \cdot \phi (x,0) dx = 0, \label{eulerweakformulation} \\ 
&\int_0^T \int_\Omega u \cdot \nabla \varphi dx dt = 0.
\end{align}
Moreover, we require that $(u \cdot n) \lvert_{\partial \Omega} = 0$ in $C([0,T]; H^{-1/2} (\partial \Omega))$ and that $\partial_n (p + (u \cdot n)^2) \lvert_{\partial \Omega} = (u \otimes u) : \nabla n$ in $L^\infty ((0,T); H^{-2} (\partial \Omega))$.
\end{definition}
We notice that, by restricting to divergence-free test functions $\phi$ in the weak formulation \eqref{eulerweakformulation}, the contribution from the pressure term vanishes. In the next proposition we will show that the above definition is equivalent to an alternative definition of weak solutions of the Euler equations as it is stated in Proposition \ref{weakformulationequivalence}.
\begin{proposition} \label{weakformulationequivalence}
Let $\alpha \in (0,1)$, and let $u \in C([0,T];C^{0,\alpha} (\Omega))$ be weakly divergence-free and satisfy $(u \cdot n) \lvert_{\partial \Omega} = 0$ as well as $ u \lvert_{t = 0} = u_0$ (where $u_0 \in C^{0,\alpha} (\Omega)$ is divergence-free). Then the weak formulation \eqref{eulerweakformulation} is equivalent to the condition that $u$ satisfies
\begin{equation}
\int_0^T \int_\Omega \bigg[u \cdot \partial_t \psi + (u \otimes u) : \nabla \psi \bigg] dx dt + \int_{\Omega} u_0 \cdot \psi (x,0) dx = 0, \label{eulerweakformulationdivergencefree}
\end{equation}
for all divergence-free $\psi \in \mathcal{D} (\Omega \times [0,T); \mathbb{R}^3)$. 

To be more precise, we have: 
\begin{enumerate}[(i)]
    \item If $(u,p)$ is a weak solution of the Euler equations in the sense of Definition \ref{eulerweaksolutiondefinition}, then $u$ satisfies weak formulation \eqref{eulerweakformulationdivergencefree}.
    \item If $u$ satisfies weak formulation \eqref{eulerweakformulationdivergencefree} and the assumptions mentioned above, then there exists a pressure $p \in L^\infty ((0,T);C^{0,\alpha} (\Omega))$ such that $(u,p)$ is a weak solution of the Euler equations in the sense of Definition \ref{eulerweaksolutiondefinition} and satisfies weak formulation \eqref{eulerweakformulation}.
\end{enumerate}
\end{proposition}
\begin{proof}
The proof of part (i) follows immediately by choosing divergence-free test functions $\phi$, which will play the role of $\psi$, in weak formulation \eqref{eulerweakformulation}.

In order to prove part (ii), we first introduce the pressure $p$, which corresponds to $u$, as the unique solution of the following elliptic boundary-value problem
\begin{equation} \label{pressureboundaryvalueproblem}
\begin{cases}
- \Delta p = (\nabla \otimes \nabla) : (u \otimes u) \quad \text{in } \Omega, \\
\partial_n (p + (u \cdot n)^2 ) = (u \otimes u) : \nabla n \quad \text{on } \partial \Omega.
\end{cases}
\end{equation}
Moreover, we require the pressure to be mean-free over $\Omega$ to guarantee uniqueness. 

As it has been mentioned already, the pressure boundary condition given in \eqref{pressureboundaryvalueproblem} (which was rigorously derived in \cite{bardos2023}) differs from the canonical one for strong solutions of the Euler equations and can be shown to be compulsory, in particular when $\alpha \in (0,1/2]$, by means of an explicitly constructed example in \cite{bardos2023} of an incompressible flow. 

It was shown in \cite{bardos2023} that there exists a unique solution $p \in L^\infty ((0,T); C^{0,\alpha} (\Omega))$  (for any given $u \in L^\infty ((0,T); C^{0,\alpha} (\Omega))$) to the boundary-value problem \eqref{pressureboundaryvalueproblem} which satisfies the following weak formulation for the pressure (see \cite[Theorem 1.5 and Definition 3.2]{bardos2023})
\begin{equation} \label{pressuredefinition}
\int_0^T \int_{\Omega} p \Delta \chi dx dt = - \int_0^T \int_{\Omega} (u \otimes u) : (\nabla \otimes \nabla) \chi dx dt,
\end{equation}
for all $\chi \in \mathcal{D} (\Omega \times (0,T); \mathbb{R})$. 

Let $\phi \in \mathcal{D} (\Omega \times [0,T); \mathbb{R}^3)$ be arbitrary, by the Leray-Helmholtz decomposition (see \cite[Theorem IV.3.5]{boyer}) we know that there exist $\psi \in C ([0,T];L^2 (\Omega))$ (such that $(\psi \cdot n)\lvert_{\partial \Omega} = 0$ and $\psi$ is weakly divergence-free) and $\chi \in C ([0,T]; H^1 (\Omega))$ such that
\begin{equation} \label{lerayhelmholtz}
\phi = \psi + \nabla \chi.
\end{equation}
By taking the divergence and the normal component at the boundary, respectively, of equation \eqref{lerayhelmholtz}, one arrives at the following boundary-value problem for $\chi$
\begin{equation} \label{gradientboundaryvalueproblem}
\begin{cases}
\Delta \chi = \nabla \cdot \phi \quad \text{in } \Omega, \\
\partial_n \chi = 0 \quad \text{on } \partial \Omega.
\end{cases}
\end{equation}
Now by using elliptic regularity estimates (see for example \cite[Theorem 1.10]{girault} and \cite{nardi}), it follows that $\chi \in C([0,T]; C^3 (\Omega))$ and hence also $\psi \in C([0,T]; C^2 (\Omega))$.

Inserting the test function $\phi$ into weak formulation \eqref{eulerweakformulation} leads to
\begin{align*}
&\int_0^T \int_\Omega \bigg[u \cdot \partial_t \phi + (u \otimes u) : \nabla \phi + p \nabla \cdot \phi \bigg] dx dt + \int_{\Omega} u_0 \cdot \phi (x,0) dx \\
&= \int_0^T \int_\Omega \bigg[u \cdot \partial_t \psi + (u \otimes u) : \nabla \psi \bigg] dx dt + \int_{\Omega} u_0 \cdot \psi (x,0) dx \\
&+ \int_0^T \int_{\Omega} \bigg[p \Delta \chi + (u \otimes u) : (\nabla \otimes \nabla) \chi \bigg] dx dt = 0,
\end{align*}
where we have used above that $\psi$, $u$ and $u_0$ are (weakly) divergence-free. The first group of terms on the right-hand side is zero because we assumed in equation \eqref{eulerweakformulationdivergencefree} that the weak formulation of the Euler equations holds for divergence-free test functions. The second group of terms vanishes because of the weak formulation for the pressure in equation \eqref{pressuredefinition}. Therefore $(u,p)$ satisfies weak formulation \eqref{eulerweakformulation}.

By Theorem 3.4 (and also Theorem 2.4) in \cite{bardos2023} it then follows that the boundary condition $\partial_n \big(p + (u \cdot n)^2 \big) \lvert_{\partial \Omega} = (u \otimes u) : \nabla n$ holds in $L^\infty ((0,T);$ $H^{-2} (\partial \Omega))$. By proceeding similarly as in the proof of Theorem 3.7 in \cite{bardos2023} one can also derive a weak formulation from \eqref{eulerweakformulation} for the case of test functions $\phi$ which do not vanish at the boundary.

It remains to show that the pressure, which satisfies \eqref{pressureboundaryvalueproblem} or \eqref{pressuredefinition}, is unique. If there would exist another mean-free pressure $\tilde{p} \in \mathcal{D}' (\Omega \times (0,T))$ one can deduce from \eqref{pressuredefinition} and \eqref{gradientboundaryvalueproblem} that
\begin{equation} \label{pressuredifference}
\langle p - \tilde{p}, \nabla \cdot \phi \rangle_{x,t} = 0,
\end{equation}
where $\langle \cdot, \cdot \rangle_{x,t}$ denotes the spatial and temporal distributional duality bracket. Therefore, \eqref{pressuredifference} implies that  $\nabla(p-\tilde{p})=0$ and consequently $\Delta (p-\tilde{p})=0$. Since the mean of $(p-\tilde{p})$ is zero we conclude that $p - \tilde{p} = 0$ and hence the pressure is unique. Therefore $(u,p)$ is a weak solution of the Euler equations in the sense of Definition \ref{eulerweaksolutiondefinition} and this completes the proof of part (ii) of the proposition.
\end{proof}
\section{Main result and its proof} \label{mainresultsection}
Before stating our main result, let us first recall the definition of the set
\begin{align}
\mathcal{H} &\coloneqq \{ \psi \in C^\infty_c (\overline{\Omega}; \mathbb{R}^3) \; \lvert \; \nabla \cdot \psi = 0, \; (\psi \cdot n) \lvert_{\partial \Omega} = 0 \}.
\end{align}
The function spaces $H$ and $W$ are then defined as the closure of the set $\mathcal{H}$ with respect to the $L^2 (\Omega)$ and the $H^1 (\Omega)$ topologies, respectively. We define $W'$ to be the dual of the Hilbert space $W$, where $W$ is endowed with the $H^1(\Omega)$ topology.
\begin{theorem} \label{absenceanomalousdissipation}
Let $\Omega$ be an open set such that $\partial \Omega \in C^4$ and $\alpha \in (0,1)$, and let $\{ u_{\nu_k} \}$ be a sequence of Leray-Hopf weak solutions of the Navier-Stokes equations \eqref{navierstokes} (for a sequence of viscosities $\nu_k \rightarrow 0$) which obey the Navier boundary conditions \eqref{boundaryconditions}-\eqref{navierboundarycondition} such that the corresponding initial data $u_{\nu_k,0}$ converge strongly in $C^{0,\alpha} (\Omega)$ to $u_0$ (as $\nu_k \rightarrow 0$). Moreover, assume that the family of solutions $\{ u_{\nu_k} \}$ satisfies the following bounds (uniformly as $\nu_k \rightarrow 0$)
\begin{equation} \label{Hölderbound}
   \sup_{k \in \mathbb{N} }
    \|u_{\nu_k} \| _{L^\infty ((0,T); C^{0,\alpha} (\Omega))} + \sup_{k \in \mathbb{N} }
    \nu_k^{1/2} \|u_{\nu_k} \| _{L^2 ((0,T); H^1 (\Omega))} <\infty.
\end{equation}
Then any weak-$*$ limit $u \in L^\infty ((0,T); L^2 (\Omega))$ of any subsequence of $\{ u_{\nu_k} \}$ is a weak solution of the Euler equations (in the sense of Definition \ref{eulerweaksolutiondefinition}), with initial data $u_0$, which belongs to $L^\infty((0,T);C^{0,\alpha} (\Omega))$ and also satisfies the impermeability boundary condition $(u \cdot n)\lvert_{\partial \Omega} = 0$.

Moreover, if $\alpha \in \left(\frac{1}{3},1 \right)$ the weak-$*$ limit $u$ above conserves energy so that
\begin{equation}
\| u (t) \|^2_{L^2(\Omega)} = \| u_0 \|^2_{L^2(\Omega)},
\end{equation}
for almost every $t \in (0,T)$. Furthermore, in this case there is no anomalous dissipation of energy for any subsequence $\{ u_{\nu_k} \}$ which converges weakly-$*$ to $u$ in $L^\infty ((0,T); L^2 (\Omega))$ (again with initial data $u_0$). That is, it then holds that (cf. \cite{bardos2013})
\begin{align} \label{vanishingdissipationnavier}
\lim_{\nu_k\rightarrow 0}  \bigg[ &2 \nu_k \int_0^T\| D (u_{\nu_k}) (t) \|^2_{L^2(\Omega)} dt \\
+ &2 \nu_k \sum_{i,j=1}^3 \int_0^T \int_{\partial \Omega} \lambda_{ij} (\nu_k,x) (u_{\nu_k})_i (u_{\nu_k})_j dx dt \bigg] = 0. \nonumber
\end{align}
Finally, there exists at least one subsequence $\{ u_{\nu_k} \}$ such that  $u_{\nu_k} \rightarrow u$ in $C ([0,T]; L^2 (\Omega))$, where the limit $u$ has the aforementioned properties.
\end{theorem}
\begin{proof}
We first show that any weak-$*$ limit $u \in L^\infty ((0,T); L^2 (\Omega))$ is a weak solution of the Euler equations (in the sense of Definition \ref{eulerweaksolutiondefinition}). Let $\{ u_{\nu_k} \}$ be a sequence converging weak-$*$ to $u$ (where we have passed to a subsequence if necessary and we have relabelled the sequence again by $\nu_k$). By using the assumed uniform boundedness \eqref{Hölderbound} of the Leray-Hopf solutions of the Navier-Stokes equations in $L^\infty ((0,T); C^{0,\alpha} (\Omega)) \cap L^\infty ((0,T); H)$, we find from equation \eqref{weakformulationnavierstokes} that $\{ \partial_t u_{\nu_k} \}$ is uniformly bounded in $L^\infty ((0,T); W')$. 

Then by applying the Aubin-Lions compactness lemma (see \cite[Theorem II.5.16]{boyer}) we deduce that there exists a subsequence $\{ u_{\nu_k} \}$ with a limit $\overline{u}$ such that $u_{\nu_k} \rightarrow \overline{u}$ in $C ([0,T]; H)$ (which allows one to make sense of the initial data). Moreover, we also have $u_{\nu_k} \rightarrow \overline{u}$ in $L^\infty ((0,T); C^{0,\beta} (\Omega))$ for any $\beta < \alpha$. The limit $\overline{u}$ belongs to $L^{\infty}((0,T); C^{0,\alpha} (\Omega))$. By the uniqueness of weak-$*$ limits (along the new subsequence) we conclude that $u = \overline{u}$ (as elements in $L^\infty ((0,T); L^2 (\Omega))$).

By passing to the limit in the weak formulation \eqref{weakformulationnavierstokes} for the Navier-Stokes equations, we find that the limit $u$ satisfies (for all $\psi \in \mathcal{D} (\Omega \times [0,T); \mathbb{R}^3) \cap L^\infty ((0,T); H)$ and $\varphi \in \mathcal{D} (\Omega \times (0,T); \mathbb{R})$)
\begin{align} 
&\int_0^T \int_\Omega \bigg[u \cdot \partial_t \psi + (u \otimes u) : \nabla \psi \bigg] dx dt + \int_{\Omega} u_0 \cdot \psi (x,0) dx = 0, \label{weakformulationdivergencefree} \\ 
&\int_0^T \int_\Omega u \cdot \nabla \varphi dx dt = 0. \nonumber
\end{align}
By the readily established regularity of $u$ and a density argument we deduce that the weak formulation \eqref{weakformulationdivergencefree} still holds for $\psi \in C^1_c ((0,T); W)$. By applying Proposition \ref{weakformulationequivalence} we therefore conclude that $u$ (and its associated pressure $p \in L^\infty ((0,T); C^{0,\alpha} (\Omega))$) is a weak solution of the Euler equations in the sense of Definition \ref{eulerweaksolutiondefinition}. 

Therefore we conclude that any weak-$*$ limit $u \in L^\infty ((0,T); C^{0,\alpha} (\Omega))$ is a weak solution of the Euler equations. Thanks to \eqref{Hölderbound} we know that by the Banach-Alaoglu theorem there exists at least one such weak-$*$ limit which is a weak solution of the Euler equations (in the sense of Definition \ref{eulerweaksolutiondefinition}).

In the case $\alpha > \frac{1}{3}$, by the results in \cite{titi2018,titi2019,bardos2023,drivas2018} (specifically Theorem 1.13 in \cite{bardos2023}) on the `conservation of energy part' in the Onsager conjecture in bounded domains, we conclude that any such weak-$*$ limit $u$ conserves energy, i.e. for almost all $t \in (0,T)$ we have
\begin{equation}
\| u (\cdot, t) \|^2_{L^2(\Omega)} = \|u_0\|^2_{L^2(\Omega)}.
\end{equation}
Now for any subsequence $\{ u_{\nu_k} \}$ such that $u_{\nu_k} \overset{\ast}{\rightharpoonup} u$ in $L^\infty ((0,T);L^2 (\Omega))$ (again for the case $\alpha > \frac{1}{3}$) we show that there is no anomalous dissipation of energy (by adapting an argument from \cite{bardos2013}). By using the weak continuity in time of Leray-Hopf weak solutions $u_{\nu_k}$, as well as the (strong) temporal continuity of the limiting weak solution $u$ of the Euler equations, it follows that for every $t \in [0,T]$ we have $u_{\nu_k} (\cdot, t) \rightharpoonup u (\cdot, t)$ in $L^2 (\Omega)$. By virtue of the energy inequality \eqref{lerayhopf2} we therefore have
\begin{align*}
&\limsup_{\nu_k \rightarrow 0} \bigg[2 \nu_k \int_0^{T} \lVert D (u_{\nu_k}) (\cdot, t) \rVert_{L^2}^2 dt + 2 \nu_k \sum_{i,j=1}^3 \int_0^{T} \int_{\partial \Omega} \lambda_{ij} (\nu_k,x) (u_{\nu_k})_i (u_{\nu_k})_j dx dt \bigg] \\
&\leq \frac{1}{2} \limsup_{\nu_k \rightarrow 0} \bigg[ \lVert u_{\nu_k,0} \rVert_{L^2}^2 - \lVert u_{\nu_k} (\cdot, T) \rVert_{L^2}^2 \bigg] \leq \frac{1}{2} \bigg[ \lVert u_0 \rVert_{L^2}^2 - \lVert u (\cdot, T) \rVert_{L^2}^2 \bigg] = 0,
\end{align*}
where we have used above the facts that $u_{\nu_k,0}$ converges strongly to $u_0$ in $C^{0,\alpha} (\Omega)$ and that the limit $u$ conserves energy. Consequently, the above proves the vanishing of the viscous dissipation as given in equation \eqref{vanishingdissipationnavier}.
\end{proof}
\begin{remark}
We observe that Theorem \ref{absenceanomalousdissipation} does not assume the existence of a regular solution (with a Lipschitz vorticity) of the Euler equations. We emphasise that the assumptions of Theorem \ref{absenceanomalousdissipation} in general are not compatible with the no-slip boundary condition for the Navier-Stokes equations, which can generate a Prandtl-type boundary layer. On the other hand it is fully compatible when $u\mapsto C(u)$ is a boundary condition which does not generate a discontinuity at the boundary on the level of the velocity, i.e. such that the boundary layer is much weaker. Examples of such boundary conditions are the aforementioned vorticity or stress tensor boundary conditions. 
\end{remark}
\begin{remark}
We note that in order to rule out anomalous dissipation (and ensure that the limiting solution conserves energy) the uniform bound \eqref{Hölderbound} can be weakened to the following two assumptions (cf. \cite{titi2019,bardos2023}): 
\begin{enumerate}[(i)]
    \item There exists a $\beta > 0$ such that
    \begin{equation} 
    \sup_{k \in \mathbb{N} }
    \|u_{\nu_k} \| _{L^\infty ((0,T); C^{0,\beta} (\Omega))} <\infty.
    \end{equation}
    \item For every $\widetilde{\Omega} \subset \subset \Omega$ there exists an exponent $\alpha (\widetilde{\Omega}) > \frac{1}{3}$ such that
    \begin{equation} 
    \sup_{k \in \mathbb{N} }
    \|u_{\nu_k} \| _{L^\infty ((0,T); C^{0,\alpha (\widetilde{\Omega})} (\widetilde{\Omega}))} <\infty.
    \end{equation}
\end{enumerate}
These assumptions (together with the Leray-Hopf regularity, as given in Definition \ref{lerayhopfdefinition}) imply the absence of anomalous dissipation.
\end{remark}
\section{Concluding remarks}

In this contribution we have established a sufficient condition for the validity of the vanishing viscosity limit of the Navier-Stokes equations \eqref{navierstokes} with Navier boundary conditions \eqref{boundaryconditions}-\eqref{navierboundarycondition}. Moreover, we have established a sufficient regularity condition for the absence of anomalous energy dissipation. In particular, in contrast to previous works the regularity assumptions in Theorem \ref{absenceanomalousdissipation} do not involve any regularity hypotheses on the pressure. In addition, we have not assumed the existence of a strong solution for the Euler equations for the corresponding initial condition. 

\subsection*{Acknowledgements}
C.B. would like to acknowledge the warm and kind hospitality of the Department of Applied Mathematics and Theoretical Physics, University of Cambridge, where part of this work was completed. D.W.B. would like to acknowledge support from the Cambridge Trust and the Cantab Capital Institute for Mathematics of Information. D.W.B. and E.S.T. have benefitted from the inspiring environment of the CRC 1114 ``Scaling Cascades in Complex Systems'', Project Number 235221301, Project C09, funded by the Deutsche Forschungsgemeinschaft (DFG). Moreover, this work was also supported in part by the DFG Research Unit FOR 5528 on
Geophysical Flows.

\begin{appendices}
\section{Global existence of Leray-Hopf weak solutions for generalised Navier boundary conditions} \label{lerayhopfappendix}
In this appendix, we will provide a sketch of the proof of the global existence of Leray-Hopf weak solutions for the case of the generalised Navier boundary conditions \eqref{boundaryconditions} and \eqref{navierboundarycondition}. The proof proceeds in a similar fashion as in \cite{kelliher2025}. We first state the result in the following theorem.
\begin{theorem} \label{lerayhopfexistencethm}
Let $u_{\nu,0} \in L^2 (\Omega)$ be a divergence-free vector field such that $(u_{\nu,0} \cdot n) \lvert_{\partial \Omega} = 0$, and assume that $\lambda_{ij} \in C^0 (\partial \Omega)$ for $i,j=1,2,3$. Then there exists a Leray-Hopf weak solution $u$ (in the sense of Definition \ref{lerayhopfdefinition}) to the Navier-Stokes equations \eqref{navierstokes} with the generalised Navier boundary conditions \eqref{boundaryconditions} and \eqref{navierboundarycondition}.
\end{theorem}
\begin{proof}
We will prove the existence of Leray-Hopf weak solutions by employing the Galerkin method. We will denote the Leray projection by $\mathbb{P}$, which we recall is the $L^2$-orthogonal projection from $L^2 (\Omega)$ onto the space $H = \{ w \in L^2 (\Omega) \; \lvert \; \nabla \cdot w = 0, \; (w \cdot n) \lvert_{\partial \Omega} = 0 \}$. In the case of the generalised Navier boundary conditions \eqref{boundaryconditions} and \eqref{navierboundarycondition}, we recall that the Stokes operator $A$ is then defined as follows
\begin{equation}
A \coloneqq - \mathbb{P} \Delta,
\end{equation}
as a map from $W$ to $W'$ (where we recall that the space $W$ and its dual $W'$ were defined at the beginning of Section \ref{mainresultsection}). In particular, the Stokes operator is defined on $W$ by duality using the following relation (for $\psi_1, \psi_2 \in W$)
\begin{equation}
\langle A \psi_1, \psi_2 \rangle = (\psi_1, \psi_2)_W \coloneqq \int_\Omega D(\psi_1) : D(\psi_2) dx + \int_{\partial \Omega} (\lambda (\nu,x) \cdot \psi_1) \cdot \psi_2 dx. 
\end{equation}
By means of a generalisation of Proposition 3.2 in \cite{kelliher2025}, we know that the space $H$ has a basis $\{ w_j \}_{j=1}^\infty \subset W$ of eigenfunctions of the Stokes operator, which is orthonormal with respect to the $L^2 (\Omega)$ inner product. Moreover, the set $\{ w_j \}_{j=1}^\infty$ is also a basis for the space $W$ and is orthogonal with respect to the inner product $(\cdot, \cdot)_W$ on $W$, that was defined above. 

Now we define $P_N$ as the $L^2$-orthogonal projection of $H$ onto the space $H_N \coloneqq \Span \{w_1, \ldots, w_N \}$. We consider the following integral formulation of the order $N$ truncated Galerkin approximation system to equation \eqref{weakformulationnavierstokes}
\begin{align}
    &\int_{\Omega} \bigg[\partial_t u_\nu^N \cdot \Psi + \nabla \cdot (u_\nu^N \otimes u_\nu^N) \cdot \Psi + 2 \nu D (u_\nu^N) : \nabla \Psi  \bigg] dx \label{galerkinsystem} \\
    &+ 2 \nu \int_{\partial \Omega} (\lambda (\nu,x) \cdot u_\nu^N) \cdot \Psi dx = 0, \quad u_\nu^N \lvert_{t=0} = P_N u_{\nu,0} \nonumber
    \end{align}
    for all $\Psi \in H_N$. By the Picard-Lindelöf theorem, there exists a local-in-time solution to the Galerkin system \eqref{galerkinsystem}. By choosing $\Psi = u_\nu^N$ and using the fact that $\nabla \cdot u_\nu^N = 0$, we deduce the following energy equality
    \begin{equation} \label{lerayhopfestimate}
    \frac{1}{2} \frac{d}{d t} \lVert u_\nu^N (\cdot, t) \rVert_{L^2}^2 + 2 \nu \lVert D(u_\nu^N) (\cdot, t) \rVert_{L^2}^2 + 2 \nu \int_{\partial \Omega} \lambda (\nu,x) : (u_\nu^N \otimes u_\nu^N) dx = 0.
    \end{equation}
    Now, by using the fact that the matrix field $\lambda$ is nonnegative semidefinite, we conclude that the term $2 \int_{\partial \Omega} \lambda (\nu,x) : (u_\nu^N \otimes u_\nu^N) dx$ above is nonnegative and hence can be removed, which yields 
    \begin{equation} \label{lerayhopfinequality}
    \frac{1}{2} \frac{d}{d t} \lVert u_\nu^N (\cdot, t) \rVert_{L^2}^2 + 2 \nu \lVert D(u_\nu^N) (\cdot, t) \rVert_{L^2}^2 \leq 0.
    \end{equation}
    Next, we recall Korn's inequality (see for example Lemma IV.7.6 in \cite{boyer}), which is given by
    \begin{equation} \label{korninequality}
    \lVert \nabla v \rVert_{L^2} \leq C \lVert D(v) \rVert_{L^2} + C \lVert v \rVert_{L^2},
    \end{equation}
    for all $v \in H^1 (\Omega)$ and for some positive constant $C$. Therefore, by virtue of estimate \eqref{lerayhopfinequality} and inequality \eqref{korninequality} we get that
    \begin{equation} \label{galerkinestimate}
    \frac{1}{2} \frac{d}{d t} \lVert u_\nu^N (\cdot, t) \rVert_{L^2}^2 + \frac{2 \nu}{C} \lVert \nabla u_\nu^N (\cdot, t) \rVert_{L^2}^2 \leq 2\nu \lVert u_\nu^N (\cdot, t) \rVert_{L^2}^2. 
    \end{equation}
    We therefore obtain that the Galerkin approximations $u_\nu^N$ are global-in-time, since \eqref{galerkinestimate} implies that
    \begin{equation}
    \lVert u_\nu^N (\cdot, t) \rVert_{L^2}^2 \leq e^{4 \nu t} \lVert u_\nu^N (\cdot, 0) \rVert_{L^2}^2 = e^{4\nu t} \lVert u_{\nu,0}^N \rVert_{L^2}^2.
    \end{equation}
    Hence the $L^2 (\Omega)$-norm of $u_\nu^N$ is uniformly bounded on any finite time interval. Furthermore, from inequality \eqref{galerkinestimate} and the above estimate we deduce the following estimate (which is uniform in $N$) for a fixed $T > 0$
    \begin{equation} \label{lerayhopfbound}
    \lVert u_\nu^N (\cdot, T) \rVert_{L^2}^2 + \frac{4 \nu}{C} \int_0^T \lVert \nabla u_\nu^N (\cdot, t) \rVert_{L^2}^2 dt \leq \big(1 + e^{4 \nu T} \big) \lVert u_{\nu,0} \rVert_{L^2}^2.
    \end{equation}
    By then proceeding in a completely analogous fashion to \cite[p. 7--8]{kelliher2025}, from equation \eqref{lerayhopfbound} we obtain
    \begin{align}
    \lVert \partial_t u_\nu^N \rVert_{L^{4/3} ((0,T); W')} &\leq C (\nu \lVert u_\nu^N \rVert_{L^2 ((0,T); W)} + \lVert u_\nu^N \rVert_{L^2 ((0,T); W)}^{3/2} \lVert u_\nu^N \rVert_{L^\infty ((0,T); H)}^{1/2} \nonumber\\
    &+ \nu \lVert \lambda \rVert_{C^0 (\partial \Omega)} \lVert u_\nu^N \rVert_{L^2 ((0,T); W)}) \nonumber \\
    &\leq \nu C (1 + \lVert \lambda \rVert_{C^0 (\partial \Omega)}) \sqrt{\frac{C}{4 \nu}}  \big(1 + e^{2 \nu T} \big) \lVert u_{\nu,0} \rVert_{L^2} \nonumber \\
    &+ C \left(\frac{C}{4 \nu}\right)^{3/4} \big(1 + e^{4 \nu T} \big) \lVert u_{\nu,0} \rVert_{L^2}^2,
    \end{align}
    where the constant $C$ depends on $T$. Then by using the Banach-Alaoglu theorem and the Aubin-Lions lemma, we know that there exists a subsequence $\{ u_\nu^N \}_{N=1}^\infty$ (which we do not relabel) which converges weakly to $u$ in $L^2 ((0,T); H^1 (\Omega))$, and weakly-$*$ to $u$ in $L^\infty ((0,T); L^2 (\Omega))$ and strongly in $L^2 ((0,T); L^2 (\Omega))$ as $N \rightarrow \infty$. In particular, we obtain the strong convergence of the sequence $u_\nu^N (\cdot,t)$ to $u_\nu (\cdot, t)$ in $L^2 (\Omega)$ for almost every $t \in [0,T]$. In addition, we have $u \in C_w([0,T]; L^2 (\Omega))$. Moreover, we recall the trace inequality (see Theorem III.2.19 in \cite{boyer})
    \begin{equation}
    \lVert v \rVert_{L^2 (\partial \Omega)}^2 \lesssim \lVert v \rVert_{L^2 (\Omega)} \lVert v \rVert_{H^1 (\Omega)},
    \end{equation}
    for all $v \in H^1 (\Omega)$. It is then clear that the sequence $\{ u_\nu^N \}_{N=1}^\infty$ converges strongly to $u$ in $L^2 ((0,T); L^2 (\partial \Omega))$. It then follows in a straightforward manner that one can pass to the limit $N \rightarrow \infty$ in the following weak formulation (for divergence-free $\Psi \in \mathcal{D} ([0,T) \times \overline{\Omega} ; \mathbb{R}^3)$ with $(\Psi \cdot n) \lvert_{\partial \Omega} = 0$)
    \begin{align}
    &\int_0^T \int_{\Omega} \bigg[u_\nu^N \cdot \partial_t P_N \Psi + (u_\nu^N \otimes u_\nu^N) : P_N \nabla \Psi - 2 \nu D (u_\nu^N) : P_N \nabla \Psi  \bigg] dx dt \\
    &+ \int_{\Omega} u_{\nu,0}^N \cdot P_N \Psi (x,0) dx - 2 \nu \int_0^T \int_{\partial \Omega} (\lambda (\nu,x) \cdot u_\nu^N) \cdot P_N \Psi dx dt = 0, \nonumber
    \end{align}
    and obtain the validity of equation \eqref{weakformulationnavierstokes} in the limit $N \rightarrow \infty$.
    Subsequently, by integrating equation \eqref{lerayhopfestimate} in time, we find
    \begin{align}
    \frac{1}{2} \lVert u_\nu^N (\cdot, t) \rVert_{L^2}^2 &+ 2 \nu \int_0^t \lVert D(u_\nu^N) (\cdot, t') \rVert_{L^2}^2 dt' + 2 \nu \int_0^t \int_{\partial \Omega} \lambda (\nu,x) : u_\nu^N \otimes u_\nu^N dx dt' \nonumber \\
    &= \frac{1}{2} \lVert u_{\nu,0}^N \rVert_{L^2}^2. \label{galerkinenergyinequality}
    \end{align}
    Then finally, by using the aforementioned convergence results of the sequence $\{ u_\nu^N \}_{N=1}^\infty$ (in particular the weak-$*$ convergence in $L^\infty ((0,T);L^2 (\Omega))$ and the weak convergence in $L^2 ((0,T); H^1 (\Omega))$ together with the weak continuity of the limit), by taking the limit $N \rightarrow \infty$ in inequality \eqref{galerkinenergyinequality}, we obtain the energy inequality \eqref{lerayhopf2}, by following similar arguments as for the three-dimensional Navier-Stokes equations with Dirichlet boundary conditions (see, e.g., \cite{constantinbook}).
\end{proof}
\end{appendices}

\footnotesize
\bibliographystyle{acm}
\bibliography{main}
\end{document}